\documentclass[11pt]{article}
\usepackage{mathrsfs}
\usepackage{amsfonts}
\usepackage{amsmath}
\usepackage{amssymb}
\usepackage{amsopn}
\usepackage{amsthm}
\usepackage{amstext}

\def\ep{\epsilon}
\def\de{\delta}

\def\pk{C^*(p_1(\ep),\cdots, p_{k}(\ep))}

\def\aa{\mathcal{A}}
\def\oo{\mathcal{O}}
\def\mm{\mathrm{M}}

\def\Pgen{\mathrm{Pgen}}
\newtheorem{lemma}{Lemma}[section]
\newtheorem{theorem}[lemma]{Theorem}
\newtheorem{corollary}[lemma]{Corollary}
\newtheorem{proposition}[lemma]{Proposition}
\newtheorem{definition}[lemma]{Definition}
\newtheorem{remark}[lemma]{Remark}

\newtheorem{example}[lemma]{Example}

\begin{document}
\date{}
\title{$C^*$-algebras generated by three projections}
\author{Shanwen Hu\thanks{E-mail: swhu@math.ecnu.edu.cn}\hspace*{2mm} and\
 Yifeng Xue\thanks{E-mail: yfxue@math.ecnu.edu.cn; Corresponding author} \\
 Department of mathematics and Research Center for Operator Algebras\\
 East China Normal University, Shanghai 200241, P.R. China
}

\maketitle

\begin{abstract} \noindent
In this short note, we prove that for a $C^*$-algebra $\aa$ generated by $n$  elements, $\mm_{k}(\tilde{\aa})$
is generated by $k$ mutually unitarily equivalent and almost mutually orthogonal projections for any
$k\ge \de(n)=\min\big\{k\in\mathbb N\,\vert\,(k-1)(k-2)\ge 2n\big\}$. Then combining this result with recent
works of Nagisa, Thiel and Winter on the generators of $C^*$--algebras, we show that for a $C^*$-algebra $\aa$
generated by finite number of elements, there is $d\ge 3$ such that $\mm_d(\tilde A)$ is generated by three mutually
unitarily equivalent and almost mutually orthogonal projections. Furthermore, for certain separable purely infinite
simple unital $C^*$--algebras and $AF$--algebras, we give some conditions that make them be generated by three
mutually unitarily equivalent and almost mutually orthogonal projections.

\vspace{3mm}

\end{abstract}

\section{Introduction}
Let $H$ be a separable complex Hilbert space with $\dim H=\infty$. Let $P$ and $Q$ be two
(orthogonal) projections on $H$. Put $M=PH$ and $N=QH$. Due to
Halmos \cite{Ha}, $P$ and $Q$ are in generic position if
$$
M\cap N=\{0\},\ M\cap N^\perp=\{0\},\ M^\perp\cap N=\{0\},\ M^\perp\cap N^\perp=\{0\}.
$$
Then the unital $C^*$--algebra generated by two projections $P$ and $Q$, which are in generic position, is
$*$--isomorphic to $\{f\in\mm_2(C(\sigma((P-Q)^2))\vert\, f(0), f(1)\ \text{are diagonal}\}$ (cf. \cite[Theorem 1.1]{V}).
Furthermore,  by \cite[Theorem 1.3]{RS}, the the universal $C^*$--algebra $C^*(p,q)$ generated by two projections $p$
and $q$ is $*$--isomorphic to the $C^*$--algebra
$$
\{f\in M_2(C([0,1]))\,\vert\, f(0),\ f(1)\ \text{are diagonal}\}
$$
which is of Type I. But in the general case of the $C^*$--algebra generated by a finite set of orthogonal projections
(at least three projections), the situation becomes unpredictable. For example, Davis showed in \cite{Da} that there
exist three projections $P_1$, $P_2$ and $P_3$ on $H$ such that the von Neumann algebra $W^*(P_1, P_2, P_3)$ generated
by $P_1$, $P_2$ and $P_3$ coincides with $B(H)$ of all bounded linear operators acting on $H$. Furthermore, Sunder
proved in \cite{Su} that for each $n\ge 3$, there exist $n$ projections $P_1,\cdots,  P_n$ on $H$ such that the von
Neumann algebra $W^*(P_1,\cdots,  P_n)$ generated by $P_1,\cdots,  P_n$ is $B(H)$ and $W^*(\mathcal M)\subsetneqq B(H)$,
whenever $\mathcal M\subsetneqq\{P_1,\cdots,P_n\}$, where $W^*(\mathcal M)$ is the von Neumann algebra generated by all
elements in $\mathcal M$.

Therefore investigating the $C^*$--algebra generated by $n$ $(n\ge 3)$ projections is an interesting topic.
Shulman studied the universal $C^*$--algebras generated by $n$ projections $p_1,\cdots, p_n$ subject to the relation
$p_1+\cdots + p_n =\lambda 1$, $\lambda\in\mathbb R$ in \cite{Sh}. She gave some conditions to make these $C^*$--algebras
type I, nuclear or exact and proved that among these $C^*$--algebras, there is a continuum of mutually
non--isomorphic ones. Meanwhile, Vasilevski considered the problem in \cite{V} that given finite set of (orthogonal)
projections $P,\, Q_1,\cdots,  Q_n$ on $H$ with the conditions
\begin{align}
\label{1a} Q_j&Q_k=\delta_{j,k}Q_k,\qquad j, k =1,\cdots,  n,\quad Q_1+\cdots+Q_n=I,\\
\label{1b} PH\cap&\, (Q_kH)^\perp=\{0\},\ Q_kH\cap (PH)^\perp=\{0\},\ k=1,\cdots, n.
\end{align}
Then what is the $C^*$--algebra $C^*(Q,P_1,\cdots, P_n)$ generated $Q,P_1,\cdots, P_n$? One of interesting
results concerning this problem is Corollary 4.5 of \cite{V}, which can be described as follows.

Let $\aa$ be a finitely generated C*-algebra with identity in $B(H)$ and let $n_0$ be a minimal number of self--adjoint
elements generating $\aa$. Then for each $n>n_0$, there exist projections $P,\, Q_1,\cdots,  Q_n$ on $H$ satisfying
(\ref{1a}) and (\ref{1b}) such that $M_n(\aa)$ is $*$--isomorphic to $C^*(P,Q_1,\cdots, Q_n)$.

Inspired by above works, we study the problem: find least number of projections in the matrix algebra of a given finitely
generated $C^*$--algebra such that these projections generates this $C^*$--algebra in this short note. The main results
of the paper are the following:

Let $\aa=C^*(a_1,\cdots,a_n)$ be the $C^*$--algebra generated by elements $a_1,\cdots, a_n$. Let $\tilde{\aa}$ denote
the $C^*$--algebra obtained by adding the unit $1$ to $\aa$ (if $\aa$ is non--unital) and let $\mm_k(\tilde\aa)$ denote
the algebra of all $n\times n$ matrices with entries in $\tilde\aa$. Then

(1) for any $k\ge\de(n)=\min\big\{k\in\mathbb N\,\vert\,(k-1)(k-2)\ge 2n\big\}$, $\mm_{k}(\tilde{\aa})$ is generated by
$k$ mutually unitarily equivalent and almost mutually orthogonal projections (see Theorem \ref{th1}).

(2) for every $l\ge\{\sqrt{n-1}\}$ and $k\ge 3$, $\mm_{kl}(\tilde {\aa})$ is generated by $k$ mutually unitarily equivalent and
almost mutually orthogonal projections  (see Proposition \ref{prop1}), where $\{x\}$ stands for the least natural number
that is greater than or equal to the positive number $x$.

\section{The main result}

In this section, we will give our main result (1) mentioned in \S 1. Firstly, we have

\begin{lemma}\label{lem1}
Let $\aa$ be a $C^*$-algebra with unit $1$ and $B_{ij}\in\aa$, for any $1\le i<j\le k$.
Suppose that $\eta=\max\{\|B_{ij}\|\,\vert\,1\le i<j\le k\}<\frac{1}{2(k-1)}$, then
$$
 T=\left[\begin{array}{ccccc}
 1&\ B_{12}&\cdots&\ B_{1k}
 \\ B_{12}^* &1& \cdots&B_{2k}
 \\\cdots&\cdots&\cdots&\cdots
 \\B_{1k}^*& B^*_{2k}&\cdots& 1
\end{array}\right]
$$
is invertible and positive, and
$$
\|T-1_k\|\le (k-1)\eta,\quad \|T^{-1/2}-1_k\|\le 2(k-1)\eta,
$$
where $1_k$ is the unit of $\mm_k(\aa)$.
\end{lemma}
\begin{proof}
 By the definition of the norm of $\mm_k(\tilde\aa)$, $\|A\|=\|[\pi(A_{ij})]_{k\times k}\|$, for $A=[A_{ij}]_{k\times k}
 \in\mm_k(\tilde\aa)$, where $\pi$ is any faithful representation of $\tilde\aa$ on a Hilbert space $K$ (see \cite{HL}),
 we may assume that $\tilde\aa\subset B(K)$ and the identity operator on $K$ is the unit of $\tilde\aa$. So $T\in B(K_k)$,
 where $K_k=\underbrace {K\oplus\cdots\oplus K}_k$.

 For any $\lambda<1-(k-1)\eta$, set
$$
A=\left[\begin{array}{cccc}1-\lambda&-\|B_{12}\|&\cdots&
-\|B_{1k}\|\\
-\|B_{12}\|&1-\lambda& \cdots & -\|B_{2k}\|\\\vdots&& &\\
-\|B_{1k}\|&-\|B_{2k}\|&\cdots& 1-\lambda\end{array}\right].
$$
Since for any $i$,  $\sum\limits_{i\not=j}\|B_{ij}\|<1-\lambda$, it follows from Levy--Dedplanques Theorem in Matrix
Analysis (see \cite{RR}) that $A$ is positive and invertible. So the quadratic form
$$f(x_1,x_2,\cdots,x_k)=\sum_{i=1}^kx_i^2-2\sum\limits_{1\le i<j\le k}\|B_{ij}\|x_ix_j$$
is positive definite and consequently, there exits $\de>0$ such that for any $(x_1,\cdots,x_k)\in{\mathbb R^n}$,
$f(x_1,\cdots,x_k)\ge \de\big(\sum\limits_{i=1}^k x_i^2\big)$.

Now for any $\xi=(\xi_1,\cdots,\xi_n)\in K_k$, we have
\begin{align*}
<(T-\lambda 1_k)\xi,\xi>&=\sum\limits_{i=1}^k\|\xi_i\|^2+\sum\limits_{1\le i<j\le k}
\Big(<B_{ij}\xi_i,\xi_j>+<B_{ij}^*\xi_j,\xi_i>\Big)\\
&\ge\sum\limits_{i=1}^k\|\xi_i\|^2-2\sum\limits_{1\le i<j\le k}\|B_{ij}\|\|\xi_i\|\|\xi_j\|\\
&=f(\|\xi_1\|,\cdots,\|\xi_k\|)\ge\de(\sum_{i=1}^k\|\xi_i\|^2)
\end{align*}
by above argument. Thus, $T-\lambda 1_k$ is invertible. Similarly, for any $\lambda >1+(k-1)\eta$, $T-\lambda 1_k$ is
also invertible.

Let $\sigma(T)$ denote the spectrum of $T$. Then we have
$$
\sigma(T)\subset[1-(k-1)\eta,1+(k-1)\eta]\subset (0,2),
$$
This indicates that $T$ is positive and invertible. Finally, by the Spectrum Mapping Theorem,
$\sigma(1_k-T)\subset [-(k-1)\eta,(k-1)\eta]$ and
\begin{align*}
\sigma(1_k-T^{-1/2})&\subset[1-(1-(k-1)\eta)^{-1/2},1-(1+(k-1)\eta)^{-1/2}]\\
&\subset[-2(k-1)\eta,2(k-1)\eta].
\end{align*}
So $\|T-1_k\|\le (k-1)\eta$ and $\|T^{-1/2}-1_k\|\le 2(k-1)\eta.$
\end{proof}

\begin{definition}\label{def1}
We say that a unital $C^*$--algebra $\mathcal E$ is generated by $n$ $(n\ge 2)$ mutually unitarily equivalent and almost
mutually orthogonal  projections if for any given $\ep>0$, there exist projections $p_1,\cdots, p_n$ in $\mathcal E$
satisfying following conditions:
\begin{enumerate}
\item[$(1)$] $p_1+\cdots+p_n$ is invertible in $\mathcal E$,
\item[$(2)$] $C^*(p_1,\cdots, p_n)=\mathcal E$ and
\item[$(3)$] for any $i\not=j$, $p_i$ is unitarily equivalent to $p_j$ in $\mathcal E$ and $\|p_ip_j\|<\ep$.
\end{enumerate}
\end{definition}

Now we present one of our main results as follows.
\begin{theorem}\label{th1}
Suppose that the $C^*$-algebra $\aa$ is generated $n$ elements $a_1,\cdots, a_n$. Then for each $k\ge
\de(n)=\min\big\{k\in\mathbb N\,\vert\,(k-1)(k-2)\ge 2n\big\}$, $\mm_{k}(\tilde{\aa})$ is generated by $k$ mutually
unitarily equivalent and almost mutually orthogonal projections.
\end{theorem}

\begin{proof}
We assume that $\aa$ is non--unital. If $\aa$ is unital, $\tilde\aa=\aa$.
Without loss generality, we may assume that $\|a_i\|=1$, $i=1,\cdots, n$.
Furthermore, we can assume $n=\frac{(k-1)(k-2)}{2}$. Otherwise, for any $n<i\le \frac{(k-1)(k-2)}{2}$,
put $a_i=1$, where $1$ is the unit of $\tilde{\aa}$.

Rewrite $\{a_1,\cdots,a_n\}=\{B_{ij}:1\le i<j\le k-2\}$ (for $\de(n)\ge 3$) and define
 $$
 T_\ep=\left[\begin{array}{ccccc}
 1&\ep B_{12}&\cdots&\ep B_{1,k-1}&\ep 1
 \\\ep B_{12}^* &1 & \cdots&\ep B_{2,k-1}&\ep 1
 \\\cdots&\cdots&\cdots&\cdots&\cdots
 \\\ep B_{1,k-1}^*&\ep B^*_{2,k-1}&\cdots& 1& \ep 1
 \\\ep 1&\ep 1&\cdots&\ep 1& 1
\end{array}\right],
\quad\forall\,\ep\in (0,1/8(k-1)).
$$

Using the canonical matrix units $\{e_{ij}\}$ for $\mm_k(\mathbb C)$, we have
\begin{align*}
T_\epsilon&=\sum_{i=1}^{k}\big(1\otimes e_{ii})+\sum_{i=1}^{k-1}(\ep 1\otimes e_{i,k}+\ep 1\otimes e_{k,i}\big)\\
&\ +\sum_{1\le i<j\le k-1}\big(\ep B_{ij}\otimes e_{ij}+\ep B_{ij}^*\otimes e_{ji}\big).
\end{align*}

By Lemma \ref{lem1}, $T_\ep$ is positive and invertible with $\|1_k-T_\ep\|\le (k-1)\ep$ and $\|1_k-T_\ep^{-1/2}\|
\le 2(k-1)\ep$.

Define $p_i(\ep)=T_\ep^{1/2}(1\otimes e_{ii})T_\ep^{1/2}$, $i=1,\cdots, k$. It is easy to verify that $p_i(\ep)$ is a
projection and $C^*(p_1(\ep),\cdots,p_k(\ep))\subset \mm_{k}(\tilde{\aa})$. In the following, we will show
$\mm_k(\tilde{\aa})\subset C^*(p_1(\ep),\cdots,p_k(\ep))$.

For all $1\le i\le k$, $p_i(\ep)\in \pk$ implies $T_\ep=\sum\limits_{i=1}^{k}p_i(\ep)$ is contained in
$\pk$. Then $T_\ep^{-1/2}\in\pk$ by Gelfand's Theorem (cf. \cite[Theorem 1.5.10]{X}), which implies that
for any $1\le i\le k$,
$$
1\otimes e_{ii}=T_\ep^{-1/2}p_i(\ep)T_\ep^{-1/2}\in\pk.
$$
It follows that for any $1\le i<j\le k-1,$
$$
B_{ij}\otimes e_{ij}=(1\otimes e_{ii})T_\ep(1\otimes e_{jj})\in\pk
$$
and for any $1\le i\le k-1$,
$$
1\otimes e_{ik}=(1\otimes e_{ii})T_\ep(1\otimes e_{kk})\in\pk.
$$
So $1\otimes e_{ki}=(1\otimes e_{ik})^*\in\pk$
and hence, for any $1\le i<j\le k-1$,
$$
1\otimes e_{ij}=(1\otimes e_{ii})(1\otimes e_{ik})(1\otimes e_{kj})\in\pk
$$
and
$1\otimes e_{ji}=(1\otimes e_{ij})^*\in\pk.$ Consequently, for any $1\le i<j\le k$ and $1\le m\le k$,
$$
B_{ij}\otimes e_{mm}=(1\otimes e_{mi})(B_{ij}\otimes e_{ij})(1\otimes e_{jm})\in\pk.
$$
Since for $i=1,\cdots, k$, $\tilde{\aa}\otimes e_{ii}$ is a $C^*$--algebra, we get for $1\le i\le k$,
$\tilde{\aa}\otimes e_{ii}\subset\pk$
and for $1\le i,j\le k$,
$$
\tilde{\aa}\otimes e_{ij}=(\tilde{\aa}\otimes e_{ii})(1\otimes e_{ij})\subset\pk.
$$

At last, we obtain that $ \mm_{k}(\tilde{\aa})\subset \pk$.

Put $I_i=1\otimes e_{ii}=T_\ep^{-1/2}p_i(\ep)T_\ep^{-1/2}$, $i=1,\cdots, k$. Then $\{I_1,\cdots, I_k\}$ is a family of mutually equivalent
and mutually orthogonal projections in $\pk$. Now for $1\le i,\,j\le k$, $i\not=j$,
\begin{align*}
\|p_j(\ep)-I_j\|&\le\|(1_k-T_\ep^{-1/2})p_j(\ep)\|+\|p_j(\ep)T_\ep^{-1/2}(1_k-T_\ep^{-1/2})\|<8(k-1)\ep<1\\
\|p_i(\ep)p_j(\ep)\|&\le\|p_i(\ep)(p_j(\ep)-I_j)\|+\|(p_i(\ep)-I_i)I_j\|<16(k-1)\ep.
\end{align*}
So $p_j(\ep)$ is unitarily equivalent to $I_j$ by Lemma 6.5.9 of \cite{X}, then to $p_i(\ep)$ and
$p_1(\ep),\cdots, p_k(\ep)$ are almost mutually orthogonal in $\pk$.\qed
\end{proof}

\begin{example}
{\rm
(1) Since $\mathbb C$ is generated by $\{1\}$, it follows from Theorem \ref{th1} that for any $k\ge 3$,
$\mm_k({\mathbb C})$ is generated by $k$ mutually unitarily equivalent and almost mutually orthogonal projections.

(2) Let $\mathcal B$ be a separable unital $C^*$--algebra and $\mathcal K$ be the $C^*$--algebra of compact operators
on the separable complex Hilbert space $H$. Then $\mathcal B\otimes\mathcal K$ is generated by a single element
(cf. \cite[Theorem 8]{OZ}). So $\mm_3(\widetilde{\mathcal B\otimes\mathcal K})$ is generated by $3$ mutually
unitarily equivalent and almost mutually orthogonal projections.
}
\end{example}

\begin{remark}\label{rem1}
{\rm Suppose that the $C^*$--algebra $\mathcal E$ with the unit $1_\mathcal E$ is generated by $k$ mutually unitarily
equivalent and almost mutually orthogonal projections. Then by Definition \ref{def1}, there are projections
$p_1,\cdots,p_k$ such that $\sum\limits^k_{i=1}p_i$ is invertible in $\mathcal E$, $p_1,\cdots,p_k$ are mutually unitarily
equivalent in $\mathcal E$ and $\|p_ip_j\|<1/2(k-1)$. Then by Corollary 3.8 of \cite{HX} and its proof, there exist
mutually orthogonal projections $p_1',\cdots,p_k'$ in $\mathcal E$ such that $\|p_i-p_i'\|<1$ and
$\sum\limits^k_{i=1}p_i'=1_\mathcal E$. Consequently, $p_i$ is unitarily equivalent to $p_i'$ in $\mathcal E$ by
\cite[Lemma 6.5.9 (2)]{X} and so that $p_i'$ is unitarily equivalent to $p_j'$ in $\mathcal E$, $i,j=1,\cdots,k$.

Now we use the $K$--Theory of $\mathcal E$ to describe above situations. The notations and properties of $K$--Theory of
$C^*$--algebras can be found in references \cite{HL} and \cite{X}. Let $[p_i]$ (resp. $[p_i']$) be the class of $p_i$
(resp. $[p_i']$) in $K_0(\mathcal E)$, $i=1,\cdots,k$. Then we have
$[1_\mathcal E]=[\sum\limits^k_{i=1}p_i']=\sum\limits^k_{i=1}[p_i']=k[p_1]$.
}
\end{remark}
\section{Some applications}
\setcounter{equation}{0}
Let $\aa$ be a $C^*$--algebra and let $M$ be a subset of $\aa_{sa}$. We call $M$ a generator of $\aa$ if $\aa$
is equal to the $C^*$--algebra  $C^*(M)$ generated by elements in $M$. If $M$ is finite, then we call $\aa$ finitely
generated and we define the number of generators $gen(A)$ by the minimum cardinality of $M$ which generates $\aa$.
We denote $gen(\aa)=\infty$ unless $\aa$ is finitely generated (cf. \cite{Na}). We call a $C^*$--algebra $\aa$ singly generated if
$gen(\aa)\le 2$. Indeed, if $\aa=C^*(\{x,y\})$ for $x, y\in\aa_{sa}$, then $C^*(x+i\,y)=\aa$.

\begin{lemma}{\rm{\cite[Theorem 3]{Na}}}\label{lem2}
Let $\aa$ be a unital $C^*$--algebra with $gen(\aa)\le n^2 + 1$ $(n\in\mathbb N)$. Then we have $gen(\mm_n(\aa))\le 2.$
\end{lemma}

Similar to the definition of $gen(\aa)$, we have following definition:
\begin{definition}\label{def2}
Let $\aa$ be a finitely generated unital $C^*$--algebra. We define the number $\Pgen(\aa)$ to be least integer $k\ge 2$
such that $\aa$ is generated by $k$ mutually unitarily equivalent and almost mutually orthogonal projections.

If no such $k$ exists, we set $\Pgen(\aa)=\infty$.
\end{definition}

\begin{remark}\label{rem2}
{\rm (1) There is a finitely generated unital $C^*$--algebra $\aa$ such that $\Pgen(\aa)=2$. For example, take
$\aa=\mm_2(\mathbb C)$ and projections
$$
p_1=\begin{bmatrix}1&0\\ 0&0\end{bmatrix},\quad
p_2=\begin{bmatrix}\ep&\sqrt{\ep(1-\ep)}\ \\ \sqrt{\ep(1-\ep)} &1-\ep \end{bmatrix},\ \forall\,\ep\in (0,1).
$$
Clearly, $p_1$ and $p_2$ are unitarily equivalent, $p_1+p_2$ is invertible and $\|p_1p_2\|\le\ep^{1/2}$. Moreover,
it is easy to check that $C^*(p_1,p_2)=\aa$. Thus, $\Pgen(\aa)=2$.

(2) If the unital $C^*$--algebra $\aa$ is infinite--dimensional and simple, then $\Pgen(\aa)\ge 3$. In fact, if
$\aa$ is  generated by two mutually unitarily equivalent and almost mutually orthogonal projections $p_1$ and $p_2$,
then there is a $*$--homomorphism $\pi\colon C^*(p,q)\rightarrow\aa$ such that $\pi(p)=p_1$ and $\pi(q)=p_2$. Thus,
$\aa=\pi(C^*(p,q))$ and hence $\aa$ is of Type $I$. But it is impossible since $\aa$ is infinite--dimensional
and simple.
}
\end{remark}

Now we present main result (2) mentioned in the end of \S 1.
\begin{proposition}\label{prop1}
Assume that the unital $C^*$--algebra $\aa$ is generated by $n$ self--adjoint elements. Then for any
$l\ge\{\sqrt{n-1}\}$ and $k\ge 3$, $\Pgen(\mm_{kl}(\aa))\le k$.
\end{proposition}
\begin{proof} Since $l\ge\sqrt{n-1}$ and $l^2+1\ge n\ge gen(\aa)$, it follow from Lemma \ref{lem2} that $\mm_l(\aa)$
is singly generated. In this case, $\de(1)=3$. So for any $k\ge 3$, $\mm_{kl}(\aa)=\mm_k(\mm_l(\aa))$ is generated by
$k$ mutually unitarily equivalent and almost mutually orthogonal projections Theorem \ref{th1}.\qed
\end{proof}

Since simple $A\!F$ $C^*$--algebra and the irrational rotation algebra are all singly generated by \cite{Na},
we have by Proposition \ref{prop1}:
\begin{corollary}\label{cor1}
If $\aa$ is a simple unital A\!F $C^*$--algebra or an irrational rotation algebra, then $\Pgen(\mm_3(\aa))\le 3$.
\end{corollary}

\begin{corollary}\label{coro2}
Let $X$ be a compact metric space with $\dim X\le m$. If $X$ can be embedded into $\mathbb C^m$, then
$\Pgen(\mm_{3k}(C(X)))\le 3$, where $k=\{\sqrt{2m-1}\}$. In general, $\Pgen(\mm_{3s}(C(X)))\le 3$, where
$s=\{\sqrt{2m}\}$.
\end{corollary}
\begin{proof} By \cite[Proposition 2]{Na},
$$
gen(C(X))=\min\{m\in\mathbb N\,\vert\, \text{there is an embedding of}\ X\ \text{into}\ \mathbb R^m\}.
$$
Therefore, if $X$ can be embedded into $\mathbb C^m$, then $gen(C(X))\le 2m$ and in general,
$X$ can be embedded into $\mathbb R^{2m+1}$ by \cite[Theorem III.4.2]{AN}. In this case, $gen(C(X))\le 2m+1$.

So the assertions follow from Proposition \ref{prop1}.\qed

\end{proof}

Recall that a projection $p$ in a $C^*$--algebra $\aa$ is infinite if there is a projection $q$ in $\aa$ with
$q<p$ such that $p$ and $q$ are equivalent (denoted by $p\sim q$) in the sense of Murray--von Neumann. $\aa$ is called to be purely infinite
if the closure of $a\aa a$ contains an infinite projection for every non--zero positive element $a$ in $\aa$
(cf. \cite{Cu}).

\begin{proposition}\label{prop2}
Let $\aa$ be a separable purely infinite simple $C^*$--algebra with the unit $1_\aa$. Suppose the class $[1_\aa]$ in
$K_0(\aa)$ has torsion. Let $m$ be the order of $[1_\aa]$. Then $3\le\Pgen(\aa)\le\min\{k\in\mathbb N\vert\,k\ge 3,
(k,m)=1\}$.

In particular, when $m$ has the form $m=3n-1$ or $m=3n-2$ for some $n\in\mathbb N$, $\Pgen(\aa)=3$.
\end{proposition}
\begin{proof} According to Remark \ref{rem2} (2), $\Pgen(\aa)\ge 3$.

Since $(k,m)=1$, $s,\,t\in\mathbb Z$ such that $ks-mt=1$ (cf. \cite{IR}). Let $c=s+ml$ and $d=t+kl$.
Then $kc-md=1$, $\forall\,l\in\mathbb N$. So we can choose $c,\,d\in\mathbb N$ such that $kc-md=1$. Set $r=kc$.

Since $r\equiv 1\!\mod{m}$, it follows from \cite[Lemma 1]{Xue} that there exist isometries $s_1,\cdots, s_r$ in $\aa$
such that
\begin{equation}\label{3a}
s_i^*s_j=0,\ i\not=j,\ i,\,j=1,\cdots, r\ \text{and}\ \sum\limits^r_{i=1}s_is_i^*=1_\aa.
\end{equation}
Define a linear map $\phi\colon\aa\rightarrow\mm_k(\aa)$ by $\phi(a)=[s_i^*as_j]_{r\times r}$. It is easy to check that
$\phi$ is a $*$--homomorphism and injective by using (\ref{3a}). Now let $A=[a_{ij}]_{r\times r}\in\mm_r(\aa)$ and
put $a=\sum\limits_{i,j=1}^{r}s_ia_{ij}s_j^*\in\aa$. Then $\phi(a)=A$ in terms of (\ref{3a}). Therefore, $\phi$ is a
$*$--isomorphism and $\aa$ is $*$--isomorphic to $\mm_r(\aa)$.

Now by Theorem 2.3 of \cite{TW}, $gen(\aa)\le 2$. Thus, by Proposition \ref{prop1}, for above $k\ge\de(1)=3$,
$c\ge 1$, $\mm_{kc}(\aa)$ is generated by $k$ mutually unitarily equivalent and almost mutually orthogonal projections
and consequently, $\Pgen(\aa)\le k$.

When $m$ has the form $m=3n-1$ or $m=3n-2$ for some $n\in\mathbb N$, $(3,m)=1$. In this case, $\Pgen(\aa)=3$ by
above argument.\qed
\end{proof}

\begin{example}
{\rm Let $\oo_n$ ($2\le n\le+\infty)$ be the Cuntz algebra. $\oo_n$ is a separable purely infinite simple unital
$C^*$--algebra with $K_0(\oo_n)\cong\begin{cases}\mathbb Z/(n-1)\mathbb Z,\ &2\le n<+\infty\\ \quad\mathbb Z, &n=+\infty
\end{cases}$ and the generator $[1_{\oo_n}]$ (cf. \cite{Cu}). Then we have
\begin{enumerate}
\item[(1)] $\Pgen(\oo_\infty)=+\infty$ by Remark \ref{rem1}.
\item[(2)] $\Pgen(\oo_n)=3$ if $n=3m$ or $n=3m-1$ for some $m\in\mathbb N$ by Proposition \ref{prop2}.
\item[(3)] $\Pgen(\oo_n)=\min\{k\in\mathbb N\vert\,k\ge 3,\ (k,n-1)=1\}$. In fact, Proposition \ref{prop2} shows that
$\Pgen(\oo_n)\le\min\{k\in\mathbb N\vert\,k\ge 3,\ (k,n-1)=1\}$. Now, $\Pgen(\oo_n)=m$ implies that there
is a projection $e\in\oo_n$ such that $m[e]=1$ in $K_0(\oo_n)$ by Remark \ref{rem1}. So there exists $s\in\mathbb N$ such that
$[e]=s[1_{\oo_n}]$. Then $ms-1\equiv 0\!\mod (n-1)$ and hence $(m,n-1)=1$.

For example: $\Pgen(\oo_{4})=4$,  $\Pgen(\oo_{13})=5$, $\Pgen(\oo_{211})=11$, etc..
\end{enumerate}
}
\end{example}

According to \cite{BKR}, a unital separable  $C^*$--algebra $\aa$ with the unit $1_\aa$ is approximately divisible if, for every
$x_1,\cdots,x_n\in A$ and any $\ep> 0$, there is a finite--dimensional $C^*$--subalgebra $\mathcal B$ with unit $1_\aa$
of $\aa$ such that $\mathcal B$ has no Abelian central projections and $\|x_iy-yx_i\|<\ep\|y\|$, $\forall\, 1\le i\le n$
and $y\in\mathcal B$.

\begin{proposition}\label{prop3}
Suppose that two separable and unital $C^*$--algebras $\aa$ and $\mathcal B$ satisfies following conditions:
\begin{enumerate}
\item[$(1)$] $\aa$ or $\mathcal B$ is nuclear$;$
\item[$(2)$] there is an integer $k\ge 3$ and a unital $C^*$--algebra $\mathcal C$ such that $\mathcal B\cong
\mm_k(\mathcal C);$
\item[$(3)$] $\aa\otimes\mathcal B$ is approximately divisible.
\end{enumerate}
Then $\Pgen(\aa\otimes\mathcal B)\le k$. Furthermore, if $k\equiv 0\!\mod 3$, then $\Pgen(\aa\otimes\mathcal B)\le 3$.
\end{proposition}
\begin{proof} If $\mathcal B$ is nuclear, applying \cite[Proposition 2.3.8]{HL} to $\mm_k(\mathcal C)$ , we get that
$\mathcal C$ is also nuclear since $\mathcal C$ is a hereditary $C^*$--subalgebra of $\mm_k(\mathcal C)$.

Now from $\aa\otimes\mathcal B\cong\mm_k(\aa\otimes\mathcal C)$, we get that $\aa\otimes\mathcal C$ is approximately
divisible by \cite[Corollary 2.9]{BKR}. Since every unital separable approximately divisible $C^*$--algebra
is singly generated by \cite[Theorem 3.1]{LS}, we obtain that $\aa\otimes\mathcal B$ is generated by $k$ mutually
unitarily equivalent and almost mutually orthogonal projections, by applying Proposition \ref{prop1} to
$\aa\otimes\mathcal C$.

If $k=3t$ for some $t\in\mathbb N$, then $\Pgen(\mm_{3t}(\aa\otimes\mathcal C))\le 3$ by Proposition \ref{prop1}.
Thus, $\Pgen(\aa\otimes\mathcal B)\le 3$ for $\aa\otimes\mathcal B\cong\mm_k(\aa\otimes\mathcal C)$.\qed
\end{proof}

Which type of $C^*$--algebras satisfy Condition (2) and (3) of Proposition \ref{prop3}?  For $A\!F$--algebras,
we have the following:

\begin{proposition}\label{prop4}
Let $\aa=\overline{\bigcup\limits^\infty_{n=1}\aa_n}$ be a $A\!F$--algebra with unit $1_\aa$, where $\aa_n$ is a
finite--dimensional $C^*$--algebra with the unit $1_\aa$ such that $\aa_m\subset\aa_n$, $\forall\,m\le n$, $m,\,n=1,2,
\cdots $. Assume that $\aa$ satisfies following conditions:
\begin{enumerate}
\item[$(1)$] no quotient of $\aa$ has an abelian projection, especially, $\aa$ is infinite dimensional simple$;$
\item[$(2)$] there is an integer $n\ge 3$ and an element $a$ in $K_0(\aa)$ such that $na=[1_\aa]$ in $K_0(\aa)$.
\end{enumerate}
If there is $k\ge 3$ such that $n\equiv 0\!\mod k$, then $\aa$ is generated by $k$ mutually unitarily equivalent and almost mutually orthogonal projections.
\end{proposition}
\begin{proof} By \cite[Proposition 3.4.5]{HL}, $a\in K_0(\aa)_+$ (the positive cone of $K_0(\aa)$). So we can find
a projection $p$ in $\mm_s(\aa_m)$ for some $s,\,m\in\,\mathbb N$ such that $[p]=a$ in $K_0(\aa)$. Consequently,
there are projections $p_1,\cdots,p_s$ in $\aa_m$ such that $p$ is unitarily equivalent to $\mathrm{diag}(p_1,\cdots,p_s)$
in $\mm_s(\aa_m)$. This indicates that
\begin{equation}\label{3eqb}
[\mathrm{diag}(\underbrace{p_1,\cdots,p_1}_{n},\cdots,\underbrace{p_s,\cdots,p_s}_{n})]=[1_\aa]\quad\text{in}\ K_0(\aa).
\end{equation}
Since $\mm_t(\aa)$ has the cancellation property of projections for all $t\in\mathbb N$, we have
\begin{equation}\label{3eqc}
\mathrm{diag}(\underbrace{p_1,\cdots,p_1}_{n},\cdots,\underbrace{p_s,\cdots,p_s}_{n})\sim\mathrm{diag}(1_\aa,
\underbrace{0,\cdots,0}_{ns-1})\quad \text{in}\ \mm_{ns}(\aa)
\end{equation}
by (\ref{3eqb}). Applying \cite[Lemma 3.4.2]{HL} to (\ref{3eqc}), we can find mutually orthogonal projections
$q_1,\cdots,q_{ns}$ in $\aa$ such that $q_{(i-1)s+1},\cdots,q_{is}$ are all unitarily equivalent to
$p_i$, $1\le i\le n$ in $\aa$.

Put $r_i=\sum\limits_{j=1}^s q_{(i-1)s+j}\in\aa$, $i=1,\cdots,n$. Then $r_ir_j=0$, $r_i\sim r_j$ and $[r_i]=[p]$
in $K_0(\aa)$, $i\not=j$, $i,j=1,\cdots,n$. So from $[r_1+\cdots+r_s]=[1_\aa]$ in $K_0(\aa)$, we obtain
$\sum\limits^s_{i=1}r_i=1_\aa$.

Let $v_i$ be partial isometries in $\aa$ such that $v_1=r_1$ and $r_1=v_i^*v_i$, $r_i=v_iv_i^*$, $r_iv_i=v_ir_1$ when
$2\le i\le n$. Define a linear mapping
$\psi\colon\aa\rightarrow\mm_n(r_1\aa r_1)$ by $\psi(a)=[v_i^*av_j]_{n\times n}$. In terms of $v_i^*v_j=0$, $i\not=j$,
$i,j=1,\cdots,n$ and $\sum\limits^n_{i=1}v_iv_i^*=1_\aa$, it is easy to check that $\psi$ is a $*$--isomorphism,
that is, $\aa$ satisfies Condition (2) of Proposition \ref{prop3}.

By \cite[Proposition 4.1]{BKR}, Condition (1) implies that $\aa$ is approximately divisible. So the assertion follows
from Proposition \ref{prop3}. \qed
\end{proof}

\begin{example}
{\rm Let $\mathcal B$ be a $U\!H\!F$--algebra. It is in one--one correspondence with a generalized integer, formal
products $q=\prod\limits^\infty_{j=1}p_j^{n_j}$ for some $\{n_j\}^\infty_{j=1}\subset \mathbb Z_+\cup\{+\infty\}$,
where $\{p_1,p_2,\cdots\}$ is the set of all positive prime numbers listed in increasing order.
According to \cite[7.4]{RLL}, $K_0(\mathcal B)$ is isomorphic to
$\big\{\frac{x}{\,y\,}\vert\,x\in\mathbb Z, y\in\mathbb N, q\equiv 0\!\mod y\big\}=\mathbb Z_{(q)}$
with $[1_\mathcal B]$ in correspondence with $1$, where  $q\equiv 0\!\!\mod y$ means that
$y=\prod\limits^\infty_{j=1}p_j^{m_j}$ for some $m_j\in\mathbb Z_+$ with $m_j\le n_j$, $j=1,\cdots,\infty$ and $m_j>0$
for only finitely many $j$.

Put $k=\min\{n\in\mathbb N\vert\,n\ge 3, q\equiv 0\!\mod n\}$. Clearly, there is $a\in K_0(\mathcal B)$ such that $ka=[1_\aa]$.
Thus there is a unital $C^*$--algebra $\mathcal C$ such that $\mathcal B\cong\mm_k(\mathcal C)$ (see the proof of
Proposition \ref{prop4}). Since $\mathcal B$ and $\aa\otimes\mathcal B$ are all approximately divisible
for any unital separable $C^*$--algebra $\aa$ by \cite{BKR}, it follows from Proposition \ref{prop3} that
$\mathcal B$ and $\aa\otimes\mathcal B$ are all generated by $k$ mutually unitarily equivalent and almost mutually
orthogonal projections, i.e., $\Pgen(\mathcal B)\le k$ and $\Pgen(\aa\otimes\mathcal B)\le k$.

Moreover, we have $\Pgen(\mathcal B)=\min\{n\in\mathbb N\vert\,n\ge 3, q\equiv 0\!\mod n\}$. In fact, since $\mathcal B$
is simple and infinite--dimensional, it follows from Remark \ref{rem2} that $\Pgen(\mathcal B)\ge 3$. Let
$m=\Pgen(\mathcal B)$. Then there is a projection $e$ in $\mathcal B$ such that $m[e]=[1_\mathcal B]$. Thus, there are
$x,\,y\in\mathbb Z_+$ with $q\equiv 0\!\mod y$ such that $m\frac{x}{\,y\,}=1$ and consequently, $q\equiv 0\!\mod m$.
So $\Pgen(\mathcal B)\ge\min\{n\in\mathbb N\vert\,n\ge 3, q\equiv 0\!\mod n\}$.

For example, if $\mathcal B$ is a $U\!H\!F$ algebra of Type $2^\infty$ or $3^\infty$, respectively, then
$\Pgen(\mathcal B)=4$ or $\Pgen(\mathcal B)=3$.
}
\end{example}

Finally, similar to Davis' result in \cite{Da} and Sunder' work in \cite{Su}, We have

\begin{proposition}\label{prop5}
Let $H$ be a separable infinite dimensional Hilbert space. Then  for any $k\ge 3$ there are $k$  mutually unitarily
equivalent and almost mutually orthogonal  projections $P_1,\cdots, P_k$ such that
$$
\mathcal K\subset C^*(P_1,\cdots,P_k)\subset W^*(P_1,\cdots,P_k)=B(H).
$$
\end{proposition}
\begin{proof}
Take $H=l^2$ and let $S$ be the unilateral shift on $H$. It's well--known that
$\mathcal K\subset C^*(S)\subset W^*(S)=B(H)$ (cf. \cite {HL}). Then there are $k$ mutually unitarily equivalent and
almost mutually orthogonal projections $Q_1,\cdots, Q_k$ in $\mm_k(C^*(S))$ such that
$ C^*(Q_1,\cdots, Q_k)=\mm_k(C^*(S))$ by Theorem \ref{th1}.

Choose isometry operators $S_1,\cdots, S_k$ on $H$ such that $S^*_iS_j=0$, $i\not=j$, $i,\,j=1,\cdots, k$ and
$\sum\limits^k_{i=1}S_iS_i^*=I$. Define a unitary operator $W\colon H\rightarrow\bigoplus\limits^k_{i=1}H$ by
$Wx=(S_1^*x,\cdots, S_k^*x)$, $\forall\,x\in H$. Then $W^*(\mm_k(\mathcal K))W=\mathcal K$ and
$W^*(\mm_k(B(H)))W=\mathcal B(H)$. Put $P_i=W^*Q_iW$, $i=1,\cdots, k$. Then $P_1,\cdots, P_k$ are mutually unitarily
equivalent and almost mutually orthogonal and $W^*(\mm_k(C^*(S)))W=C^*(P_1,\cdots, P_k)$.
So from
$$
\mm_k(\mathcal K)\subset C^*(Q_1,\cdots, Q_k)\subset W^*(Q_1,\cdots,Q_k)=\mm_k(B(H)),
$$
we obtain the assertion.\qed
\end{proof}

\vskip0.2cm

\noindent{\bf{Acknowledgement.}} The authors thank to Professor Huaxin Lin and the referee for their helpful comments
and suggestions.

\enddocument